\documentclass[10pt]{article}
\usepackage{amssymb}
\usepackage{amsmath}
\usepackage{amsfonts}
\usepackage{amsthm}
\usepackage{graphics,graphicx}
\usepackage{float}
\usepackage{mathtools}
\usepackage[dvipsnames]{xcolor}
\usepackage{tkz-graph}
\usetikzlibrary{shapes, arrows.meta, positioning}

\usepackage[english]{babel}
\usepackage{authblk}

\newtheorem{theorem}{Theorem}
\newtheorem{lemma}{Lemma}
\newtheorem{proposition}{Proposition}
\newtheorem{corollary}{Corollary}
\newtheorem{example}{Example}

\newtheorem{problem}{Problem}
\newtheorem{o-problem}{Open problem}

\title{On $P$-crucial square-free permutations}
\author[1]{Alexandr Valyuzhenich \thanks{Postdoctoral Research Station of Mathematics, Hebei Normal University, Shijiazhuang, China; Email address: graphkiper@mail.ru}}

\date{}
\begin{document}
\maketitle

\begin{abstract}
A permutation is square-free if it does not contain two consecutive factors of length two or more that are order-isomorphic.
A square-free permutation of length $n$ is $P$-crucial, where $P$ is a subset of $\{0,1,\ldots,n\}$, 
if any of its extensions in any position from the set $P$ contains a square.
In 2015, Gent, Kitaev, Konovalov, Linton and Nightingale initiated the study of $P$-crucial square-free permutations.
In particular, they showed that $\{0,1,n-1,n\}$-crucial square-free permutations of length $n$, where $n\leq 22$, exist if and only if $n=17$ or $n=21$.
In this work, we prove that for any $m\geq 2$ there exists a $\{0,1,8m+4,8m+5\}$-crucial square-free permutation of length $8m+5$.
\end{abstract}

\textbf{Keywords:} $P$-crucial permutation, square-free permutation, bicrucial permutation, right-crucial permutation

\textbf{AMS classification:} 68R15, 05A05

\section{Introduction}\label{Sec:Intro}

In 2011, Avgustinovich et al. \cite{AKPV11} introduced the concept of a square-free permutation.
They proved that the number of square-free permutations of length $n$ is $n^{n(1-\varepsilon_n)}$ where $\varepsilon_n \rightarrow 0$ when $n \rightarrow \infty$.
In the same paper, Avgustinovich et al. initiated the study of right-crucial and bicrucial square-free permutations.
In particular, they proved that there exist arbitrarily long bicrucial square-free permutations.
These investigations were continued in \cite{GKKLN15,GJ22}.
In particular, Groenland and Johnston \cite{GJ22} finished the classification of possible lengths of bicrucial square-free permutations.
Recently, Akhmejanova et al. \cite{AKVV25} proved that for any $k\geq 3$ there exist arbitrarily long bicrucial $k$-power-free permutations.
In 2015, Gent, Kitaev, Konovalov, Linton and Nightingale \cite{GKKLN15} initiated the study of $P$-crucial square-free permutations.
They showed that $\{0,1,n-1,n\}$-crucial square-free permutations of length $n$, where $n\leq 22$, exist if and only if $n=17$ or $n=21$.
In this work, we prove that for any $m\geq 2$ there exists a $\{0,1,8m+4,8m+5\}$-crucial square-free permutation of length $8m+5$.

The paper is organized as follows.
In Section \ref{Sec:Definitions}, we introduce basic definitions.
In Section \ref{Sec:Prelim}, we give preliminary results.
In Section \ref{Sec:Main}, we prove that for any $m\geq 2$ there exists a $\{0,1,8m+4,8m+5\}$-crucial square-free permutation of length $8m+5$.
In Section \ref{Sec:S-crucial}, we discuss the interconnection between S-crucial and $\{0,1,n-1,n\}$-crucial square-free permutations.
In particular, we show that these two classes of permutations do not coincide.

\section{Basic definitions}\label{Sec:Definitions}
A {\em permutation} of length $n$ is a word of length $n$ whose symbols are distinct integers.
For example, $132$ and $594$ are permutations of length $3$.
We say that a permutation $P$ is a permutation of the set $X$ if the set of symbols of $P$ is $X$.
A permutation $F$ is a {\em factor} of a permutation $P$ if $P=P_1FP_2$ for some permutations $P_1$ and $P_2$, where $P_1$ and $P_2$ may be empty.
For example, $364$ and $425$ are factors of the permutation $1364257$.
Note that any factor of a permutation is also a permutation.
A permutation $F$ is a {\em suffix} of a permutation $P$ if $P=P_1F$ for some permutation $P_1$.
A permutation $F$ is a {\em prefix} of a permutation $P$ if $P=FP_2$ for some permutation $P_2$.
For a permutation $P=p_1\ldots p_n$ and positive integers $i$ and $j$, where $1\leq i\leq j\leq n$, 
denote by $P[i,j]$ the factor $p_ip_{i+1}\ldots p_j$.

For a positive integer $n$, define $[n]$ by $[n] = \{1,\ldots,n\}$.
Two permutations $p_1\ldots p_n$ and $q_1\ldots q_n$ are called {\em order-isomorphic} 
if $(p_i-p_j)(q_i-q_j)>0$ for all $i,j\in [n]$, $i\neq j$.
For example, $1324$ is order-isomorphic to $2648$. 
Note that this relation is an equivalence relation on the set of all permutations of a fixed length.
We denote this equivalence by $\sim$.
A {\em square} is a permutation of the form $X_1X_2$, where $X_1 \sim X_2$ and $|X_1| \geq 2$.
For example, the permutation $1324$ is a square.
We say that a permutation $P$ {\em contains} a square if $P=P_1XP_2$, where $P_1$ and $P_2$ are permutations and $X$ is a square.
A permutation is called {\em square-free} if it does not contain squares.
For example, the permutation $1364257$ is square-free.
Various generalizations of squares in permutations can be found in \cite{DGR21}.

Let $P$ be a permutation of length $n$ and let $i\in \{0,1,\ldots,n\}$.
A permutation $Q$ of length $n+1$ is an {\em extension} of $P$ in position $i$
if the permutation obtained from $Q$ by deleting the ($i+1$)-th symbol is order-isomorphic to $P$.
For example, $1324$ is an extension of the permutation $123$ in position $2$.

A square-free permutation of length $n$ is called {\em $P$-crucial}, where $P$ is a subset of $\{0,1,\ldots,n\}$, 
if for every $i\in P$ any of its extensions in position $i$ contains a square.
Left-crucial, right-crucial and bicrucial square-free permutations can be defined as $P$-crucial permutations in the following way.
A square-free permutation is called {\em left-crucial} if it is $P$-crucial with $P=\{0\}$.
A square-free permutation of length $n$ is called {\em right-crucial} if it is $P$-crucial with $P=\{n\}$.
A square-free permutation of length $n$ is called {\em bicrucial} if it is $P$-crucial with $P=\{0,n\}$.
Other types of crucial and bicrucial permutations can be found in \cite{AKV12,AKT23,C24}.
A square-free permutation of length $n$ is called {\em S-crucial} if it is $P$-crucial with $P=\{0,1,\ldots,n\}$.

For integers $k$ and $m$, where $k<m$, denote $[k,m] = \{k,k+1,\ldots,m\}$.

\section{Preliminaries}\label{Sec:Prelim}

\textbf{Levels in permutations.}
Let $P=p_1\ldots p_n$ be a permutation that does not contain squares of length $4$.
Then there exists $i\in \{0,1,2,3\}$ such that for every integer $t$ and for all valid indices the inequalities

\begin{equation}\label{Eq:1}
p_{i+4t}<p_{i+4t\pm 1} \textrm{ and } p_{i+4t+2}>p_{i+4t+2\pm 1}
\end{equation}
hold (see \cite{AKPV11,GKKLN15}). 
For a permutation $P$ satisfying (\ref{Eq:1}), 
we say that the symbols of $P$ with the indices $4t+i$ form the {\em lower} level,
the symbols of $P$ with the indices $4t+i\pm 1$ form the {\em medium} level, and
the symbols of $P$ with the indices $4t+i+2$ form the {\em upper} level.
For example, the permutation $2463157$ satisfies (\ref{Eq:1}) for $i=1$.
The symbols $2$ and $1$ form its lower level, the symbols $6$ and $7$ form its upper level, and the symbols $4$, $3$ and $5$ form its medium level 
(see Figure \ref{fig:levels} below).

\begin{figure}[h!]
\centering
\begin{tikzpicture}[scale=0.63]
    \node (A1) at (6,0.5) [circle,fill,inner sep=1.5pt,label=below left:1] {};
    \node (A2) at (2,1) [circle,fill,inner sep=1.5pt,label=below left:2] {};
    \node (A3) at (5,1.5) [circle,fill,inner sep=1.5pt,label=above right:3] {};
    \node (A4) at (3,2) [circle,fill,inner sep=1.5pt,label=above left:4] {};
    \node (A5) at (7,2.5) [circle,fill,inner sep=1.5pt,label=above left:5] {};
    \node (A6) at (4,3) [circle,fill,inner sep=1.5pt,label=above right:6] {};
    \node (A7) at (8,3.5) [circle,fill,inner sep=1.5pt,label=above right:7] {};

    \draw (A2) -- (A6);
    \draw (A6) -- (A3);
    \draw (A3) -- (A1);
    \draw (A1) -- (A5);
    \draw (A5) -- (A7);
\end{tikzpicture}
\caption{A diagram of the permutation $2463157$.}
\label{fig:levels}
\end{figure}
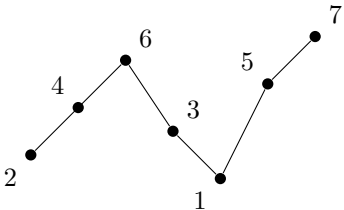

\textbf{High-medium-low construction.} 
Now we briefly recall one construction of square-free permutations proposed by Avgustinovich et al. in \cite{AKPV11}.
Let $P$ be a permutation satisfying the following conditions:

(S1) $P$ satisfies (\ref{Eq:1}) for some $i\in \{0,1,2,3\}$, 

(S2) any symbol from the lower level is less than any symbol from the medium level and any symbol from the medium level is less than any symbol from the upper level,  

(S3) symbols from the medium level form a square-free permutation.

Avgustinovich et al. \cite{AKPV11} proved that all permutations satisfying (S1)-(S3) are square-free.

\begin{lemma}[\cite{AKPV11}, Lemma 1]\label{L:H-M-L}
Any permutation constructed by the high-medium-low construction is square-free.
\end{lemma}

In what follows, we will need the following two special cases of the high-medium-low construction.

\textbf{Construction 1.}
Let $m\geq 1$.
Let $A=a_1\ldots a_m$ and $C=c_1\ldots c_m$ be permutations of length $m$ and let $B=b_1\ldots b_{2m-1}$ be a permutation of length $2m-1$.
Suppose that $A$, $B$ and $C$ satisfy the following conditions:

(A1) any symbol of $A$ is less than any symbol of $B$ and any symbol of $B$ is less than any symbol of $C$,

(A2) $B$ is square-free.

We define a permutation $P=p_1 \ldots p_{4m-1}$ of length $4m-1$ as follows:

$$
p_i=\begin{cases}
a_{\frac{i+3}{4}},&\text{if $i \equiv 1\bmod 4$;}\\
c_{\frac{i+1}{4}},&\text{if $i \equiv 3\bmod 4$;}\\
b_{\frac{i}{2}},&\text{if $i$ is even.}
\end{cases}
$$

For example, for $m=1$ and $m=2$ we get the permutations $a_1b_1c_1$ and $a_1b_1c_1b_2a_2b_3c_2$ respectively.
Taking into account (A1), we see that $P$ satisfies (\ref{Eq:1}) for $i=1$.
Therefore, $P$ satisfies (S1).
The symbols $a_1,\ldots,a_m$ form the lower level of $P$, the symbols $b_1,\ldots,b_{2m-1}$ form the medium level of $P$, 
and the symbols $c_1,\ldots,c_m$ form the upper level of $P$. 
Since $A$, $B$ and $C$ satisfy (A1) and (A2), $P$ satisfies (S2) and (S3).
Thus, $P$ satisfies (S1)-(S3).
Therefore, by Lemma \ref{L:H-M-L}, $P$ is square-free.
Thus, we have the following.

\begin{corollary}\label{Cor:Construction-1}
Any permutation constructed by Construction 1 is square-free.
\end{corollary}

\textbf{Construction 2.}
Let $m\geq 1$.
Let $U=u_1\ldots u_m$ and $W=w_1\ldots w_m$ be permutations of length $m$ and let $V=v_1\ldots v_{2m-1}$ be a permutation of length $2m-1$.
Suppose that $U$, $V$ and $W$ satisfy the following conditions:

(B1) any symbol of $U$ is greater than any symbol of $V$ and any symbol of $V$ is greater than any symbol of $W$,

(B2) $V$ is square-free.

We define a permutation $P=p_1 \ldots p_{4m-1}$ of length $4m-1$ as follows:

$$
p_i=\begin{cases}
u_{\frac{i+3}{4}},&\text{if $i \equiv 1\bmod 4$;}\\
w_{\frac{i+1}{4}},&\text{if $i \equiv 3\bmod 4$;}\\
v_{\frac{i}{2}},&\text{if $i$ is even.}
\end{cases}
$$

For example, for $m=1$ and $m=2$ we get the permutations $u_1v_1w_1$ and $u_1v_1w_1v_2u_2v_3w_2$ respectively.
Taking into account (B1), we see that $P$ satisfies (\ref{Eq:1}) for $i=3$.
Therefore, $P$ satisfies (S1).
The symbols $u_1,\ldots,u_m$ form the upper level of $P$, the symbols $v_1,\ldots,v_{2m-1}$ form the medium level of $P$, 
and the symbols $w_1,\ldots,w_m$ form the lower level of $P$. 
Since $U$, $V$ and $W$ satisfy (B1) and (B2), $P$ satisfies (S2) and (S3).
Thus, $P$ satisfies (S1)-(S3).
Therefore, by Lemma \ref{L:H-M-L}, $P$ is square-free.
Thus, we have the following.

\begin{corollary}\label{Cor:Construction-2}
Any permutation constructed by Construction 2 is square-free.
\end{corollary}

We also need the following simple fact.

\begin{proposition}\label{P:H-M-L-Factors}
Let $P$ be a permutation constructed by the high-medium-low construction. 
Then no factor of $P$ of length $4$ is order-isomorphic to $2341$, $3214$, $4123$ and $1432$.  
\end{proposition}
\begin{proof}
Let $F=xyzt$ be a factor of $P$ such that $F\sim 2341$.
Since $x<y<z$, we know that $x$, $y$ and $z$ belong to the lower, medium and upper levels of $P$ respectively.
This implies that $t$ belongs to the medium level of $P$. Therefore, $t>x$, i.e. $F \not\sim 2341$.
The proofs for other permutations are similar.
\end{proof}

In what follows, we will need the following two lemmas.

\begin{lemma}\label{L:Modifications}
The following statements hold:

\begin{enumerate}
  
  \item Let $P_1$ be a permutation constructed by Construction $1$.
  Let $x$ and $y$ be two integers such that $x>a>y$ for any symbol $a$ of $P_1$.
  Then the permutation $xP_1y$ is square-free.
  
  \item Let $P_2$ be a permutation constructed by Construction $2$.
  Let $z$ and $t$ be two integers such that $z<b<t$ for any symbol $b$ of $P_2$.
  Then the permutation $zP_2t$ is square-free.

\end{enumerate}
\end{lemma}
\begin{proof}
1. Let $P=xP_1y$. Suppose that $P$ contains a square $X=X_1X_2$.
Let us show that $X$ contains exactly one symbol from the set $\{x,y\}$.
If $X$ contains the symbols $x$ and $y$, then $X=P$.
This contradicts the fact that the length of $P_1$ is odd.
On the other hand, by Corollary \ref{Cor:Construction-1}, $P_1$ is square-free.
Therefore, $X$ contains exactly one symbol from the set $\{x,y\}$. 
Without loss of generality, we can assume that $X$ contains the symbol $y$ and does not contain the symbol $x$.
Let us consider three cases depending on the length of $X_1$.

\textbf{Case 1: $|X_1|=2$.}
In this case, we have $X_1 \sim 12$ and $X_2 \sim 21$. 
Therefore, $X_1 \not\sim X_2$, that is, $X$ is not a square.

\textbf{Case 2: $|X_1|=3$.}
In this case, we have $X_1 \sim 321$ and $X_2 \sim 231$. 
Therefore, $X_1 \not\sim X_2$, that is, $X$ is not a square.

\textbf{Case 3: $|X_1| \geq 4$.}
The suffix of $X_2$ of length $4$ is order-isomorphic to $2341$.
On the other hand, $X_1$ is a factor of $P_1$ (because $X$ does not contain the symbol $x$).
Therefore, by Proposition \ref{P:H-M-L-Factors}, no factor of $X_1$ of length $4$ is order-isomorphic to $2341$.
Hence, $X_1 \not\sim X_2$, that is, $X$ is not a square.

2. The proof for the permutation $zP_2t$ is similar.
\end{proof}

\begin{lemma}\label{L:Special-Square-Free}
For any $m\geq 2$, there exists a square-free permutation $r_1r_2 \ldots r_{2m-1}$ of length $2m-1$ 
such that $r_2>r_3$ and $r_{2m-3}>r_{2m-2}$.
\end{lemma}
\begin{proof}
Let us consider two cases depending on the parity of $m$.

Suppose that $m=2k$, where $k$ is a positive integer.
Let $P=p_1 \ldots p_{4k-1}$ be a permutation constructed by Construction $2$.
By Corollary \ref{Cor:Construction-2}, $P$ is square-free.
Since $P$ is constructed by Construction $2$, we have that $p_2>p_3$ and $p_{4k-3}>p_{4k-2}$
(we have $v_1>w_1$ and $u_k>v_{2k-1}$).
Thus, we have proved the statement for all even $m$.

Suppose that $m=2k+1$, where $k$ is a positive integer.
Let $Q$ be a permutation of length $4k+3$ constructed by Construction $1$.
Let $T=t_1 \ldots t_{4k+1}$ be the permutation obtained from $Q$ by removing the first and last symbols.
By Corollary \ref{Cor:Construction-1}, $Q$ is square-free.
Since $T$ is a factor of $Q$, $T$ is also square-free.
Since $Q$ is constructed by Construction $1$, we have that $t_2>t_3$ and $t_{4k-1}>t_{4k}$
(we have $c_1>b_2$ and $b_{2k}>a_{k+1}$).
Therefore, we have proved the statement for all odd $m$.
\end{proof}

\section{Main results}\label{Sec:Main}
The main result of this paper is the following.
\begin{theorem}
For any $m\geq 2$, there exists a $\{0,1,8m+4,8m+5\}$-crucial square-free permutation of length $8m+5$. 
\end{theorem}
\begin{proof}
Let $m\geq 2$.
We will construct a $\{0,1,8m+4,8m+5\}$-crucial square-free permutation of length $8m+5$ step by step.
In our construction we will use Lemma \ref{L:Special-Square-Free}, Constructions 1 and 2 and Lemma \ref{L:Modifications}.

By Lemma \ref{L:Special-Square-Free}, there exists a square-free permutation $R_m=r_1 \ldots r_{2m-1}$ of length $2m-1$ such that $r_2>r_3$ and $r_{2m-3}>r_{2m-2}$.
Let $Y_m=y_1 \ldots y_{2m-1}$ be a permutation of the set $[3m+4,5m+2]$ such that $Y_m\sim R_m$.
Since $Y_m\sim R_m$, we have that $Y_m$ is square-free, $y_2>y_3$ and $y_{2m-3}>y_{2m-2}$.
Let $X_m$ be a permutation of the set $[2m+4,3m+3]$ and let $Z_m$ be a permutation of the set $[5m+3,6m+2]$. 
For example, we can take $X_m=(3m+3)(3m+2)\ldots (2m+4)$ and $Z_m=(6m+2)(6m+1)\ldots (5m+3)$.
Applying Construction $1$ for $A=X_m$, $B=Y_m$ and $C=Z_m$, we construct a permutation $H'_m$ of length $4m-1$
(we can use Construction $1$ because $X_m$, $Y_m$ and $Z_m$ satisfy (A1) and (A2)).
Then, we define a permutation $H_m$ of length $4m+1$ by $H_m=(6m+3)H'_m(2m+3)$.
Using Lemma \ref{L:Modifications} for $x=6m+3$, $P_1=H'_m$ and $y=2m+3$, we obtain that $H_m$ is square-free. 
We define permutations $S_m$ and $T_m$ of length $2m+1$ by
$S_m=(6m+4)(6m+5)(6m+6)\ldots (8m+4)$ and $T_m=234\ldots (2m+2)$.  
Applying Construction $2$ for $U=S_m$, $V=H_m$ and $W=T_m$, we construct a permutation $E'_m$ of length $8m+3$
(we can use Construction $2$ because $S_m$, $H_m$ and $T_m$ satisfy (B1) and (B2)).
Finally, we define a permutation $E_m$ of length $8m+5$ by $E_m=1E'_m(8m+5)$ (see Example \ref{Ex:1} for $m=2$).
Using Lemma \ref{L:Modifications} for $z=1$, $P_2=E'_m$ and $t=8m+5$, we obtain that $E_m$ is square-free.

Let us prove that $E_m$ is $\{0,1,8m+4,8m+5\}$-crucial.
Suppose that $Q=q_1 \ldots q_{8m+6}$ is an extension of $E_m$ in position $i$, where $i \in \{0,1,8m+4,8m+5\}$.
Let us prove that $Q$ contains a square.
Let us consider four cases.

\textbf{Case 1: $i=8m+5$.}
Let us consider two subcases depending on the value of $q_{8m+6}$.

\textbf{Case 1.1: $q_{8m+6}>q_{8m+5}$.}
Since $Q$ is an extension of $E_m$ in position $8m+5$, we have
$Q[8m-1,8m+2] \sim y_{2m-1}(2m+1)z_m(8m+4)$ and
$Q[8m+3,8m+5] \sim (2m+3)(2m+2)(8m+5)$. 
In addition, we have $3m+4 \leq y_{2m-1}<z_m \leq 6m+2$ and $q_{8m+6}>q_{8m+5}$.
Therefore, the permutations $Q[8m-1,8m+2]$ and $Q[8m+3,8m+6]$ are order-isomorphic.
Hence, the suffix of $Q$ of length $8$ is a square.

\textbf{Case 1.2: $q_{8m+6}<q_{8m+5}$.}
Since $Q$ is an extension of $E_m$ in position $8m+5$, we have
$q_{8m+3}q_{8m+4} \sim (2m+3)(2m+2)$.
On the other hand, we have $q_{8m+5}>q_{8m+6}$.
Therefore, the suffix of $Q$ of length $4$ is a square.

\textbf{Case 2: $i=8m+4$.}
Let us consider four subcases depending on the value of $q_{8m+5}$.

\textbf{Case 2.1: $q_{8m+5}>q_{8m+6}$.}
Since $Q$ is an extension of $E_m$ in position $8m+4$, we have
$q_{8m+3}q_{8m+4} \sim (2m+3)(2m+2)$.
In addition, we have $q_{8m+5}>q_{8m+6}$.
Therefore, the suffix of $Q$ of length $4$ is a square.

\textbf{Case 2.2: $q_{8m+5}<q_{8m+4}$.}
Since $Q$ is an extension of $E_m$ in position $8m+4$, we have
$q_{8m+2}q_{8m+3} \sim (8m+4)(2m+3)$.
In addition, we have $q_{8m+4}>q_{8m+5}$.
Therefore, the permutation $Q[8m+2,8m+5]$ is a square.

\textbf{Case 2.3: $q_{8m+3}<q_{8m+5}<q_{8m+6}$.}
Since $Q$ is an extension of $E_m$ in position $8m+4$, we have
$Q[8m-1,8m+2] \sim y_{2m-1}(2m+1)z_m(8m+4)$ and
$q_{8m+3}q_{8m+4}q_{8m+6} \sim (2m+3)(2m+2)(8m+5)$. 
In addition, we have $3m+4 \leq y_{2m-1}<z_m \leq 6m+2$ and $q_{8m+3}<q_{8m+5}<q_{8m+6}$.
Therefore, the permutations $Q[8m-1,8m+2]$ and $Q[8m+3,8m+6]$ are order-isomorphic.
Hence, the suffix of $Q$ of length $8$ is a square.

\textbf{Case 2.4: $q_{8m+4}<q_{8m+5}<q_{8m+3}$.}
Since $Q$ is an extension of $E_m$ in position $8m+4$, we have
$$Q[8m-10,8m-3] \sim (8m+1)y_{2m-3}(2m-1)z_{m-1}(8m+2)y_{2m-2}(2m)x_m$$ and
$$Q[8m-2,8m+4] \sim (8m+3)y_{2m-1}(2m+1)z_m(8m+4)(2m+3)(2m+2).$$
In addition, we have $y_{2m-3}>y_{2m-2}$ and $q_{8m+4}<q_{8m+5}<q_{8m+3}$.
Therefore, the permutations $Q[8m-10,8m-3]$ and $Q[8m-2,8m+5]$ are order-isomorphic.
Hence, $Q$ contains a square of length $16$.

\textbf{Case 3: $i=0$.} The proof for this case is similar to the proof for the case $i=8m+5$.

\textbf{Case 4: $i=1$.} The proof for this case is similar to the proof for the case $i=8m+4$.
\end{proof}

In the following example we discuss in detail our main construction for $m=2$.

\begin{example}\label{Ex:1}
There is a unique permutation $Y_2=y_1y_2y_3$ of the set $[10,12]$ such that $y_2>y_3$ and $y_1>y_2$.
Namely, we have $Y_2=(12)(11)(10)$.
Since $X_2$ is a permutation of the set $[8,9]$ and $Z_2$ is a permutation of the set $[13,14]$, 
we can take $X_2=98$ and $Z_2=(14)(13)$.
Construction $1$ for $A=X_2$, $B=Y_2$ and $C=Z_2$ gives the permutation $H'_2=9(12)(14)(11)8(10)(13)$ of length $7$.
Then, we define a permutation $H_2$ of length $9$ by $H_2=(15)H'_2 7$.
Therefore, we have $H_2=(15)9(12)(14)(11)8(10)(13)7$ (see Figure \ref{F:1}).
We define permutations $S_2$ and $T_2$ of length $5$ by
$S_2=(16)(17)(18)(19)(20)$ and $T_2=23456$.
Construction $2$ for $U=S_2$, $V=H_2$ and $W=T_2$ gives the permutation 
$E'_2=(16)(15)29(17)(12)3(14)(18)(11)48(19)(10)5(13)(20)76$ of length $19$.
Finally, we define a permutation $E_2$ of length $21$ by $E_2=1E'_2(21)$.
Therefore, we have $E_2=1(16)(15)29(17)(12)3(14)(18)(11)48(19)(10)5(13)(20)76(21)$ (see Figure \ref{F:2}).
\end{example}

\begin{figure}[h!]
\centering
\begin{tikzpicture}[scale=0.65]
    \node (A1) at (1,9) [circle,fill,inner sep=1.5pt,label=above left:15] {};
    \node (A2) at (2,3) [circle,fill=red,inner sep=1.5pt,label=below left:9] {};
    \node (A3) at (3,6) [circle,fill=green,inner sep=1.5pt,label=above left:12] {};
    \node (A4) at (4,8) [circle,fill=blue,inner sep=1.5pt,label=above left:14] {};
    \node (A5) at (5,5) [circle,fill=green,inner sep=1.5pt,label=above left:11] {};
    \node (A6) at (6,2) [circle,fill=red,inner sep=1.5pt,label=below left:8] {};
    \node (A7) at (7,4) [circle,fill=green,inner sep=1.5pt,label=above left:10] {};
    \node (A8) at (8,7) [circle,fill=blue,inner sep=1.5pt,label=above left:13] {};
    \node (A9) at (9,1) [circle,fill,inner sep=1.5pt,label=below left:7] {};

    \draw (A1) -- (A2);
    \draw (A2) -- (A3);
    \draw (A3) -- (A4);
    \draw (A4) -- (A5);
    \draw (A5) -- (A6);
    \draw (A6) -- (A7);
    \draw (A7) -- (A8);
    \draw (A8) -- (A9);
\end{tikzpicture}
\caption{A diagram of the permutation $H_2$. 
The symbols of the permutations $X_2$, $Y_2$ and $Z_2$ are highlighted in red, green and blue respectively.}
\label{F:1}
\end{figure}
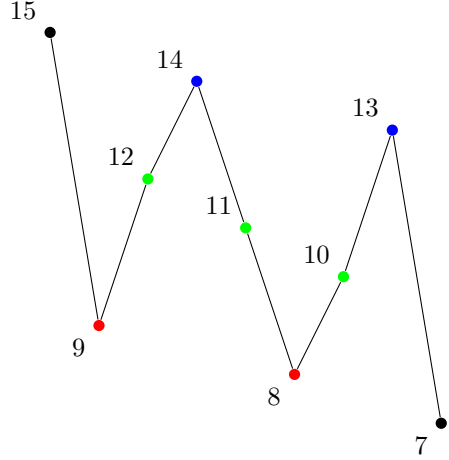

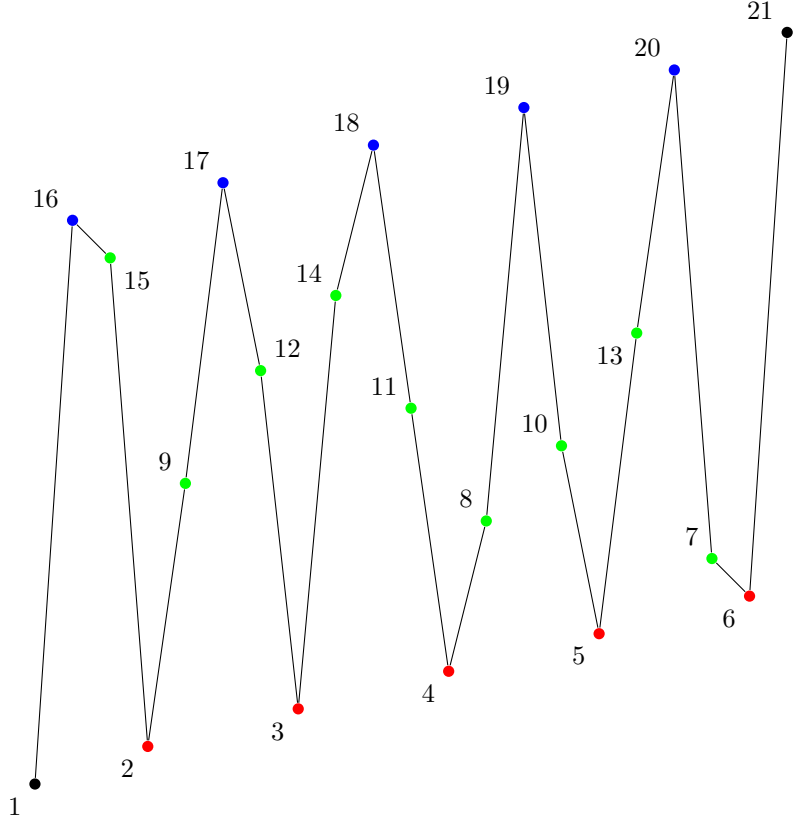
\begin{figure}[h!]
\centering
\begin{tikzpicture}[scale=0.5]
    \node (A1) at (1,1) [circle,fill,inner sep=1.5pt,label=below left:1] {};
    \node (A2) at (2,16) [circle,fill=blue,inner sep=1.5pt,label=above left:16] {};
    \node (A3) at (3,15) [circle,fill=green,inner sep=1.5pt,label=below right:15] {};
    \node (A4) at (4,2) [circle,fill=red,inner sep=1.5pt,label=below left:2] {};
    \node (A5) at (5,9) [circle,fill=green,inner sep=1.5pt,label=above left:9] {};
    \node (A6) at (6,17) [circle,fill=blue,inner sep=1.5pt,label=above left:17] {};
    \node (A7) at (7,12) [circle,fill=green,inner sep=1.5pt,label=above right:12] {};
    \node (A8) at (8,3) [circle,fill=red,inner sep=1.5pt,label=below left:3] {};
    \node (A9) at (9,14) [circle,fill=green,inner sep=1.5pt,label=above left:14] {};
    \node (A10) at (10,18) [circle,fill=blue,inner sep=1.5pt,label=above left:18] {};
    \node (A11) at (11,11) [circle,fill=green,inner sep=1.5pt,label=above left:11] {};
    \node (A12) at (12,4) [circle,fill=red,inner sep=1.5pt,label=below left:4] {};
    \node (A13) at (13,8) [circle,fill=green,inner sep=1.5pt,label=above left:8] {};
    \node (A14) at (14,19) [circle,fill=blue,inner sep=1.5pt,label=above left:19] {};
    \node (A15) at (15,10) [circle,fill=green,inner sep=1.5pt,label=above left:10] {};
    \node (A16) at (16,5) [circle,fill=red,inner sep=1.5pt,label=below left:5] {};
    \node (A17) at (17,13) [circle,fill=green,inner sep=1.5pt,label=below left:13] {};
    \node (A18) at (18,20) [circle,fill=blue,inner sep=1.5pt,label=above left:20] {};
    \node (A19) at (19,7) [circle,fill=green,inner sep=1.5pt,label=above left:7] {};
    \node (A20) at (20,6) [circle,fill=red,inner sep=1.5pt,label=below left:6] {};
    \node (A21) at (21,21) [circle,fill,inner sep=1.5pt,label=above left:21] {};
    \draw (A1) -- (A2);
    \draw (A2) -- (A3);
    \draw (A3) -- (A4);
    \draw (A4) -- (A5);
    \draw (A5) -- (A6);
    \draw (A6) -- (A7);
    \draw (A7) -- (A8);
    \draw (A8) -- (A9);
    \draw (A9) -- (A10);
    \draw (A10) -- (A11);
    \draw (A11) -- (A12);
    \draw (A12) -- (A13);
    \draw (A13) -- (A14);
    \draw (A14) -- (A15);
    \draw (A15) -- (A16);
    \draw (A16) -- (A17);
    \draw (A17) -- (A18);
    \draw (A18) -- (A19);
    \draw (A19) -- (A20);
    \draw (A20) -- (A21);
\end{tikzpicture}
\caption{A diagram of the permutation $E_2$.
The symbols of the permutations $S_2$, $H_2$ and $T_2$ are highlighted in blue, green and red respectively.}
\label{F:2}
\end{figure}

\pagebreak

\section{S-crucial and $P$-crucial permutations}\label{Sec:S-crucial}
In this section, we discuss the interconnection between S-crucial and $\{0,1,n-1,n\}$-crucial square-free permutations.

Firstly, we discuss one observation about S-crucial and $\{0,1,n-1,n\}$-crucial square-free permutations from \cite{GKKLN15}
(see the first paragraph of Section 2.5 in \cite{GKKLN15}).
It is obvious that any S-crucial square-free permutation of length $n$ is $\{0,1,n-1,n\}$-crucial.
Gent et al. \cite{GKKLN15} claimed that the opposite is also true.
Their proof was based on the following fact:

\begin{itemize}
  \item Let $P$ be a square-free permutation of length $n$.
For every $i\in [2,n-2]$, any extension of $P$ in position $i$ contains a square of length $4$.
\end{itemize}

Let us show that this statement is not true for $i=2$ and $i=n-2$.
Let us consider the permutation $M=135426$. Note that $M$ is square-free.
The permutation $M'=1465237$ is an extension of $M$ in position $4$.
In addition, $M'$ is square-free ($M'$ can be constructed by Construction 1).
Therefore, the statement is not true for $i=n-2$. The counterexample for $i=2$ is similar.
On the other hand, the statement is true for all $i\in [3,n-3]$.
Therefore, the observation can be easily corrected as follows.

\begin{proposition}
A square-free permutation of length $n$, where $n\geq 5$, is S-crucial if and only if it is $\{0,1,2,n-2,n-1,n\}$-crucial.
\end{proposition}

Secondly, we discuss Theorem 7 from \cite{GKKLN15}.
In the proof of Theorem 7, Gent et al. claimed that the permutation $P=24315(11)(10)69(12)87(13)(17)(15)(14)(16)$ of length $17$
is S-crucial square-free. In fact, $P$ is $\{0,1,16,17\}$-crucial.
Let us show that $P$ is not S-crucial.
Let us consider the permutation $Q=24315(11)(10)69(12)87(13)(17)(15)0(14)(16)$ of length $18$.
Note that $Q$ is an extension of $P$ in position $15$.
In addition, $Q$ is square-free (it is enough to check all factors of $Q$ of length $8$ or $16$ containing $0$). 
Therefore, $P$ is not S-crucial. 
Thus, it is not obvious that there exist S-crucial square-free permutations of length $17$.
On the other hand, we know that there are no S-crucial square-free permutations of length at most $16$.
In this context, it is interesting to consider the following problem.

\begin{problem}
Are there S-crucial square-free permutations of length at least $17$?
\end{problem}

Finally, we discuss our main construction in the context of S-crucial permutations.
We have proved that the permutation $E_m$ is $\{0,1,8m+4,8m+5\}$-crucial square-free.
Note that $E_m$ is not S-crucial.
Indeed, let $Q$ be the permutation obtained from $E_m$ by inserting the symbol $0$ in position $8m+3$.
It is not difficult to prove that $Q$ is square-free.
Therefore, $E_m$ is not S-crucial.
On the other hand, any extension of $E_m$ in position $i$, where $i\in [3,8m+2]$, contains a square of length $4$.

\end{document}